\documentclass[11pt]{amsart}
\usepackage{lmodern}
\usepackage[a4paper,margin=3cm]{geometry}
\usepackage[utf8]{inputenc}
\usepackage[T1]{fontenc}
\usepackage[english]{babel}
\usepackage{graphicx}
\usepackage{xcolor}
\usepackage[all,cmtip]{xy}
\usepackage{cancel}
\usepackage{ulem}
\SelectTips{lu}{11}
\usepackage{setspace}
\usepackage{subfigure,graphics,graphicx,psfrag}
\usepackage{tikz} 
\usepackage{overpic}
\usepackage{srcltx}
\usepackage{amsmath,amsfonts,amssymb,mathrsfs}

\theoremstyle{plain}
\newtheorem{theorem}{Theorem}[section]

\theoremstyle{remark}
\newtheorem{remark}{Remark}

\theoremstyle{definition}
\newtheorem{definition}[theorem]{Definition}

\title{Persistent Leray's spectral sequence}

\author[E.\ L.\ dos Santos]{Edivaldo L.\ dos Santos}
\address{Departamento de Matem\'{a}tica \\
	Federal University of S\~ao Carlos \\
	Rodovia Washington Luiz. km 235 \\
	S\~ao Carlos, SP, Brazil. \textit{E-mail address:} \rm edivaldo@ufscar.br}

    \author[ Telmo A. ]{Telmo I. Acosta Vellozo }
\address{ CENUR Noreste, Universidad de la Rep\'ublica - Uruguay \textit{E-mail address:} \rm telmo.acosta@cur.edu.uy}


\thanks{This work is partially supported by the Projeto Tem\'atico: Topologia Alg\'ebrica,Geom\'etrica e Diferencial, FAPESP Process Number 2022/16455-6. This study was financed in part by the Coordenação de Aperfeiçoamento de Pessoal de Nível Superior - Brasil (CAPES) - Finance Code 001.}

\begin{document}

\maketitle

\noindent{\bf Abstract.}
In this work, we construct a persistent version of the well-known Leray spectral sequence. More precisely, we construct a spectral sequence that computes the persistent cohomology of a space from the persistent cohomology in each open set and its intersections with a covering that is the pre-image under a function of a covering of a known space.

\section{Introduction}

Persistent homology/cohomology is an algebraic method for measuring the topological
features of shapes and functions. In the last 20 years, the research and application of persistent homology/cohomology have been intense; for an introduction  and survey, see \cite{Carlsson}.

\subsection{Persistent Cohomology}

In this section,  we show that the persistent cohomology of a filtered simplicial complex is a graded module over a polynomial ring, as Zomorodian and Carlsson had done for persistent homology in \cite{Zomorodian}.

Given a filtered complex %indexed in decreasing order
$$\emptyset\subseteq K_n\subseteq K_{n-1}\subseteq\cdots\subseteq K_1\subseteq K_0=K,$$
 for generality, we let $K_i=K$ for all $i\leq 0$ and we let $K_i=\emptyset$ for all $i> n$. We call $K$ a \textit{filtered complex}. The $i$-th complex $K_i$ has associated the coboundary operator $d_i^k$ and groups $C_i^k,\ Z_i^k,\ B_i^k$ and $H_i^k$ for $i,k\geq 0$, and we have the restriction homomorphisms $\nu_i^k:H_i^k\longrightarrow H_{i+1}^k$. Then we have the sequence
$$\xymatrix{\cdots\ar[r]^1& H_{-1}^\ast\ar[r]^1 &H_0^\ast \ar[r]^{\nu_0^\ast} & H_1^\ast\ar[r]^{\nu_1^\ast} &\cdots\ar[r]^{\nu_{n-2}^\ast}& H_{n-1}^\ast\ar[r]^{\nu_{n-1}^\ast}& H_n^\ast\ar[r]&0}.$$

\begin{definition}
The $p$-persistent $k$-th cohomology group is
$$H_{i,p}^k=Z_i^k/(B_{i+p}^k\cap Z_i^k).$$
\end{definition}

\begin{remark}
    
 If we define $\nu_{i,p}^\ast$ as the restriction homomorphism  from $H_i^\ast$ to $H_{i+p}^\ast$, then $H_{i,p}^\ast$ is the image of $\nu_{i,p}^\ast$.
\end{remark}

\begin{definition}
A persistence module $\mathcal{M}$ is a family of $R$-modules $M_i$, together with homomorphisms $\varphi_i:M_i\longrightarrow M_{i+1}$.
\end{definition}

For example, the sequence of cohomology of the filtration is a persistence module, where $\varphi_i$ are the restrictions.

\begin{definition}
A persistence module $\{M_i,\varphi_i\}$ is of finite type if each component module is a finitely generated $R$-module, and if the maps $\varphi_i$ are isomorphisms for $i\geq m$ and $ i\leq n$ for some integers $m$ and $n$.
\end{definition}

\subsection{Correspondence}

Suppose we have a persistence module $\mathcal{M}=\{M_i,\varphi_i\}_{i\in \mathbb{Z}}$ over a ring $R$. We now equip $R[t]$ with the standard grading and define a graded module over $R[t]$ by
$$\alpha(\mathcal{M})= \bigoplus_{i\in \mathbb{Z}} M_i,$$
where the $R$-module structure is simply the sum of the structures on the individual components, and where the action of $t$ is given by
$$t\cdot (\dots,m_{-1},m_0,m_1,m_2,\dots)=(\dots,m_{-1},\varphi_0(m_0),\varphi_1(m_1),\varphi_2(m_2),\dots).$$
That is, $t$ simply shifts elements of the module up in the graduation.\\
\par In the case of the persistence module coming from a cohomology of the filtration of a complex $K$, we use the following notation $PH^\ast(K)$ to $\alpha(K)$ and we can write
$$PH^\ast(K)= \bigoplus_{i\in \mathbb{Z}} H_{i}^\ast.$$

\begin{theorem}
The correspondence $\alpha$ defines an equivalence of categories between the category of persistence modules of finite type over $R$ and the category of finitely generated graded modules over $R[t]$.
\end{theorem}
The proof follows from  Artin-Rees theory in commutative algebra \cite{Eisenbud}.

\section{Leray's theorem}

\'Alvaro Torras Casas, in his doctoral thesis \cite{Torras}, explored the Mayer-Vietoris spectral sequence associated with filtered covers on filtered complexes. From this study, he developed an algorithm to compute persistent homology.

In this section, we will explore Leray's spectral sequence associated with filtered covers on filtered complexes. The difference between Mayer-Vietoris spectral sequences and Leray's spectral sequence is that, while Mayer-Vietoris spectral sequence works on the homology of the open cover of the space, Leray's spectral sequence works on the cohomology of preimage under a map of an open cover of a known space. There are at least two advantages to working with Leray's spectral sequence instead of the Mayer-Vietoris spectral sequence. The first is that the algorithms to calculate persistent cohomology are faster, and the second is that it works with open covers of known spaces.

Consider $f:X\longrightarrow Y$ a map from one space to another, we want to study how the cohomology groups of $X$ relate to those of $Y$. Let $\mathcal{U}$ be any cover of $Y$ such that $f^{-1}\mathcal{U}$ is a locally finite  cover of $X$. By Generalized Mayer-Vietoris principle 
$$H_d^\ast(X) \simeq H_D\{C_\delta^\ast(f^{-1}\mathcal{U},C_d^\ast)\},$$
where the double complex $C_\delta^\ast(f^{-1}\mathcal{U},C_d^\ast)$ consists of the $p$-cochains of the cover $\mathcal{U}$ with values in the $q$-cochains, the horizontal maps is the difference operators $\delta$ and the vertical ones are the coboundary morphisms $d$. The differential operator $D$ is given by $D=\delta+ (-1)^pd$ in the single complex $C^n(f^{-1}\mathcal{U})=\bigoplus_{p+q=n}C_\delta^\ast(f^{-1}\mathcal{U},C_d^\ast)$.\\
By spectral sequence of double complex, if $K$ is the double complex $C_\delta^\ast(f^{-1}\mathcal{U},C_d^\ast)$ on $X$, then the spectral sequence of $K$ converges to 
$$ H_D\{C_\delta^\ast (f^{-1}\mathcal{U},C_d^\ast)\}$$
and
$$E_2^{p,q}= H_\delta^{p}H_d^q\{C_\delta^\ast(f^{-1}\mathcal{U}, C_d^\ast)\}.$$%=H_\delta^p(U_{\alpha_0\dots\alpha_p},H_d^q(f^{-1}(U_{\alpha_0\dots\alpha_p}))). $$
Thus we have the next version of Leray's theorem.
\begin{theorem}[Leray's Theorem]\label{leray}
Let $X$ and $Y$ be spaces, let $\mathcal{U}$ be a cover of $Y$ such that $f^{-1}\mathcal{U}$ is a cover for $X$ locally finite and let $f:X\longrightarrow Y$ be a map. Then there is a spectral sequence converging to $H_d^\ast(X)$ and $E_2^{p,q}=H_\delta^{p}H_d^q\{C_\delta^\ast(f^{-1}\mathcal{U}, C_d^\ast)\}$.
\end{theorem}
Another way to prove the Theorem \ref{leray} is to take the vertical filtration, since the construction was done using horizontal filtration and the page two of spectral sequence of double complex to vertical filtration stabilizes.

Note that $H_d^q(f^{-1}\mathcal{U})$ is a presheaf on $Y$, we will use the notation $\mathscr{H}^q$ for this presheaf. Then in the Leray's theorem, $E_2^{p,q}=H_\delta^p(\mathcal{U},\mathscr{H}^q)$ is the \v Cech cohomology of the cover $\mathcal{U}$ with values in $\mathscr{H}$.\\

We are now in a position to state the main theorem of this section.

\begin{theorem}[Leray's Theorem for Persistence Module]\label{Leray teo pm}
Let $X$ and $Y$ be topological spaces, let $f:X\rightarrow Y$ be a map, let $\mathcal{U}$ be a cover of $Y$ such that $f^{-1}\mathcal{U}$ is a cover for $X$ locally finite and a filtration $X=K_0\supseteq K_1\supseteq\cdots\supseteq K_N=\emptyset$. Then there is a spectral sequence converging to $PH^\ast(X)$ and $E_2^{p,q}=H_\delta^p(\mathcal{U}, PH^q(f^{-1}\mathcal{U}))$.
\begin{proof}

 Let $X=K_0\supseteq K_1\supseteq\cdots\supseteq K_N=\emptyset$ be a filtration of $X$, then for the presheaf ${H}^q_d(K_i\cap f^{-1}(U))$, we will use the notation $\mathscr{H}_i^q$ and the notation $E_{2,i}^{p,q}$ for $H_\delta^p(\mathcal{U},\mathscr{H}_i^q)$. As the Leray's theorem is true for the restriction of the map $f$ to each subspace $K_i$, we have a spectral sequence converging to $H_d^\ast(K_i)$ whose page $E_2$ is $E_{2,i}^{p,q}$. If we consider $\bigoplus_i E_{2,i}^{p,q}=\bigoplus_i H_\delta^p(\mathcal{U},\mathscr{H}_i^q) =H_\delta^p (\mathcal{U},\bigoplus_i \mathscr{H}_i^q)$. For the other side the direct limit commute with the direct sum, then we have a spectral sequence converging to $\bigoplus_i H_d^\ast(K_i)=PH^\ast(X)$ whose page $E_2$ is $H_\delta^p(\mathcal{U},\bigoplus_i\mathscr{H}_i^q)=H_\delta^p(\mathcal{U}, PH^q(f^{-1}\mathcal{U}))$.\newline\newline
Now we will to show that the restriction homomorphism $\eta_i^\ast :C_i^\ast\longrightarrow C_{i+1}^\ast$ induces a homomorphism on the direct sum of Leray's sequence of filtration, for this we need to show:
\begin{enumerate}
    \item $d\circ \eta_i^q =\eta_i^{q+1}\circ d$. That is the same that next diagram is commutative 
    $$\xymatrix{C_{i+1}^q(f^{-1}(U_{\alpha_0\dots\alpha_p}))\ar[r]^d & C_{i+1}^{q+1}(f^{-1}(U_{\alpha_0\dots\alpha_p})) \\ C_i^q(f^{-1}(U_{\alpha_0\dots\alpha_p}))\ar[u]^{\eta_i^q}\ar[r]^d & C_i^{q+1}(f^{-1}(U_{\alpha_0\dots\alpha_p})).\ar[u]^{\eta_i^{q+1}}} $$
    Let $\tau\in C_i^q(f^{-1}(U_{\alpha_0\dots\alpha_p}))$ and $z=(z_0,\dots ,z_{q+1})\in C_{q+1}^{i+1}(f^{-1}(U_{\alpha_0\dots\alpha_p}))$, then
    $$\eta_i^{q+1}\circ d(\tau)(z)=d(\tau)(\iota_{q+1}^{i+1}(z)),$$
    where $\iota_{q+1}^{i+1}: C_{q+1}^{i+1}(f^{-1}(U_{\alpha_0\dots\alpha_p}))\longrightarrow C_{q+1}^i(f^{-1}(U_{\alpha_0\dots\alpha_p}))$ is the inclusion, then
    $$\eta_i^{q+1}\circ d(\tau)(z)=\sum_{i=0}^p (-1)^i \tau\iota_{q+1}^{i+1}(z_0,\dots,\hat{z}_i,\dots,z_{q+1}), $$
    and for the other side
    $$d\circ\eta_i^q(\tau)(z)=\tau\iota_q^{i+1}(\sum_{i=0}^p (-1)^i(z_0,\dots,\hat{z}_i,\dots,z_{q+1}))=\sum_{i=0}^p(-1)^i \tau\iota_q^{i+1}(z_0,\dots,\hat{z}_i,\dots,z_{q+1}). $$
    Then $d\circ\eta_i^q =\eta_i^{q+1}\circ d$.
    \item $\delta\circ\eta_i^q = \eta_i^{q}\circ \delta$. That is the same that next diagram is commutative
    $$\xymatrix{\prod_{\alpha_0<\dots<\alpha_p}C_{i+1}^q(f^{-1}(U_{\alpha_0\dots\alpha_p}))\ar[r]^\delta & \prod_{\alpha_0<\dots<\alpha_{p+1}}C_{i+1}^q(f^{-1}(U_{\alpha_0\dots\alpha_{p+1}})) \\ \prod_{\alpha_0<\dots<\alpha_p}C_i^q(f^{-1}(U_{\alpha_0\dots\alpha_p}))\ar[u]^{\eta_i^q}\ar[r]^\delta & \prod_{\alpha_0<\dots<\alpha_{p+1}}C_i^q(f^{-1}(U_{\alpha_0\dots\alpha_{p+1}})).\ar[u]^{\eta_i^q}} $$
    Let $\tau\in \prod_{\alpha_0<\dots<\alpha_p}C_i^q(f^{-1}(U_{\alpha_0\dots\alpha_p}))$ and $z\in C_i^q(f^{-1}(U_{\alpha_0\dots\alpha_{p+1}}))$, then
    $$\delta\circ\eta_i^q(\tau)(z)= \sum_{i=0}^{p+1}(-1)^i(\eta_i^q(\tau))_{\alpha_0\dots\hat{\alpha}_i\dots\alpha_{p+1}}(z)= \sum_{i=0}^{p+1}(-1)^i (\tau \iota_q^{i+1})_{\alpha_0\dots\hat{\alpha}_i\dots\alpha_{p+1}}(z),$$
    and for the other side
    $$\eta_i^q\circ \delta(\tau)(z)=\delta(\tau)(\iota_q^{i+1}(z))=\sum_{i=0}^{p+1} (-1)^i \tau_{\alpha_0\dots\hat{\alpha}_i\dots\alpha_{p+1}}(\iota_q^{i+1}(z)) ,$$
    as $(\tau\iota_q^{i+1})_{\alpha_0\dots\hat{\alpha}_i\dots\alpha_{p+1}}(z)= \tau_{\alpha_0\dots\hat{\alpha}_i\dots\alpha_{p+1}}(\iota_q^{i+1}(z)$, we have $\delta\circ\eta_i^q= \eta_i^q\circ \delta$.
    \item $\eta$ respects the filtration (horizontal filtration of the double complex). This is true because the restriction acting in the index of the summation, not in index $p$ or $q$. 
\end{enumerate}
We will call $\eta_{r,i}^{p,q}:E_{r,i}^{p,q}\rightarrow E_{r,i+1}^{p,q}$ the homomorphism  induced by $\eta_i^\ast$ and we define $$\eta_{r,i}^n:\bigoplus_{p+q=n}E_{r,i}^{p,q}\rightarrow \bigoplus_{p+q=n} E_{r,i+1}^{p,q}$$ as $\eta_{r,i}^n= \oplus_{p+q=n}\eta_{r,i}^{p,q}$. Now we can define an action on the spectral sequence $\bigoplus_{p+q=n}\bigoplus_i E_{2,i}^{p,q}=\bigoplus_{p+q=n}H_\delta^p(\mathcal{U},\bigoplus_i\mathscr{H}_i^q)$
given by
$$t(\dots ,m_{r,0}^{n},m_{r,1}^{n},m_{r,2}^{n}\dots) =(\dots ,\eta_{r,0}^{n}(m_{r,0}^{n}),\eta_{r,1}^{n}(m_{r,1}^{n}),\eta_{r,2}^{n}(m_{r,2}^{n}),\dots). $$
where $m_{r,i}^n=\oplus_{p+q=n}m_{r,i}^{p,q}$.

Then this action defines an action on the limit of the spectral sequence, the limit is $\bigoplus_i H^\ast(K_i)$ and we have that the induced homomorphism is isomorphic to the restriction $\eta_i^\ast:H_d^\ast(K_i)\longrightarrow H_d^\ast(K_{i+1})$ taking the vertical filtration on the double complex. Then we conclude that the spectral sequence $\bigoplus_i E_{2,i}^{p,q}$ converges to $\alpha(\{H_d^\ast(K_i),\eta_i^\ast\})=PH^\ast(X)$, which completes the proof of the theorem.
\end{proof}
\end{theorem}


\addcontentsline{toc}{chapter}{Bibliography}

\begin{thebibliography}{25}\singlespace
 

\bibitem{Bauer} BAUER, U.; KERBER, M.; REININGHAUS, J. \textbf{Distributed computation of persistent homology.} Society for Industrial and Applied Mathematics, 2014, p. 31-38.


\bibitem{Bott} BOTT, R.; TU, L. W. \textbf{Differential forms in algebraic topology.} Springer-Verlag New York, 1982.



\bibitem{Marco} DE FREITAS CONTESSOTO, M. A. \textbf{Some Persistent Cohomology Invariants and an Axiomatic Version of Persistent Homology.} PH. D. Math. Thesis, UNESP, 2021.



\bibitem{Edelsbrunner} EDELSBRUNNER, H.; HARER, J. \textbf{Computational Topology, An Introduction.} Amer. Math. Soc., 2010.

\bibitem{Eisenbud} EISENBUD, D. \textbf{Commutative Algebra with a Wiew Toward Algebraic Theory.} Springer New York, 1995.



\bibitem{Carlsson} CARLSSON, G. \textbf{Topology and data} Bull. Amer. Math., v. 46, n. 2, 2009, p. 249-274.



\bibitem{Torras} TORRAS, A. \textbf{Persistence Spectral Sequences.} A thesis submitted for the degree of Doctor of Philosophy, Cardiff University, 2022.


\bibitem{Zomorodian} ZOMORODIAN, A.; CARLSSON, G.  \textbf{Computing Persistent Homology.} Discrete and Computational Geometry, v. 33, 2005, p. 249-274.

\end{thebibliography}
\end{document}